\documentclass[12pt,reqno]{amsart}
\usepackage{amsmath,amsthm,amssymb,amsfonts,amscd}
\usepackage{mathrsfs}
\usepackage{bbm}
\usepackage{bbding}
\usepackage{graphicx,latexsym}
\usepackage[backref=page]{hyperref}
\usepackage{hyperref}
\usepackage{geometry}
\usepackage{color}
\usepackage{xcolor}
\usepackage{picture,epic}
\usepackage{tikz}

\numberwithin{equation}{section}

\theoremstyle{plain}
\newtheorem{theorem}{Theorem}[section]

\newtheorem{proposition}[theorem]{Proposition}

\theoremstyle{definition}

\theoremstyle{remark}

\renewcommand{\Im}{\operatorname{Im}}
\newcommand{\vol}{\operatorname{vol}}

\newcommand{\sym}{\operatorname{sym}}

\newcommand{\X}{\mathbb{X}}

\newcommand{\SL}{\operatorname{SL}}

\newcommand{\dd}{\mathrm{d}}

\newcommand{\Ma}{\mathcal{M}}

\makeatletter
\def\@tocline#1#2#3#4#5#6#7{\relax
  \ifnum #1>\c@tocdepth 
  \else
    \par \addpenalty\@secpenalty\addvspace{#2}%
    \begingroup \hyphenpenalty\@M
    \@ifempty{#4}{%
      \@tempdima\csname r@tocindent\number#1\endcsname\relax
    }{%
      \@tempdima#4\relax
    }%
    \parindent\z@ \leftskip#3\relax \advance\leftskip\@tempdima\relax
    \rightskip\@pnumwidth plus4em \parfillskip-\@pnumwidth
    #5\leavevmode\hskip-\@tempdima
      \ifcase #1
       \or\or \hskip 1em \or \hskip 2em \else \hskip 3em \fi%
      #6\nobreak\relax
    \hfill\hbox to\@pnumwidth{\@tocpagenum{#7}}\par
    \nobreak
    \endgroup
  \fi}
\makeatother

\begin{document}

\title[On the $L^6$-norm of holomorphic Hecke eigenforms]
{On the $L^6$-norm of holomorphic Hecke eigenforms}

\author{Chengliang Guo}
\address{Mathematical Research Center\\ Shandong University \\ Jinan \\ Shandong 250100 \\China}
\email{chengliang.guo@mail.sdu.edu.cn}
\bigskip

\author{Liangxun Li}
\address{Mathematical Research Center\\ Shandong University \\ Jinan \\ Shandong 250100 \\China}
\address{HUN-REN Alfr\'ed R\'enyi Institute of Mathematics \\ Budapest\\ Pf. 127, H-1364 \\ Hungary}
\email{lilxmath@163.com}
\date{\today}

\begin{abstract}
Let $H_k$ be an $L^2$-normalized Hecke basis for the space of all holomorphic cusp forms of weight $k$.
We show that
$
\max_{f\in H_k}\Vert F\Vert_6\gg (\log\log k)^{\frac{1}{2}}
$
where $F(z)=(\Im z)^{\frac{k}{2}}f(z).$
This confirms that the $L^6$-norm of Hecke eigenforms does not converge uniformly as the weight goes to infinity.
We also give some results on the joint mass of degree $6$.
\end{abstract}

\keywords{$L^p$-norm, Hecke eigenforms}

\subjclass[2020]{11F11}


\maketitle

\section{Introduction}
Let $\mathbb{H}$ be the upper half-plane in $\mathbb{C}$ equipped with the hyperbolic measure $\dd\mu(z)=\dd x\dd y/y^2,$ and let $\X=\SL_2(\mathbb{Z})\backslash\mathbb{H}$ be the hyperbolic surface obtained as the quotient of $\mathbb{H}$ by the modular group $\SL_2(\mathbb{Z})$.
By the classical theory of modular forms, the space of all holomorphic cusp forms of weight $k$ on $\X$ admits an orthogonal basis $H_k$ consisting of $L^2$-normalized holomorphic Hecke eigenforms of weight $k$. 
For $f\in H_k$, as the weight $k$ tends to infinity, a fundamental problem is to understand the value distribution of $f$.
One approach to this problem is to examine the growth of the $L^p$-norms for $2<p\leq \infty$. Let $F(z)=(\Im z)^{\frac{k}{2}}f(z)$ and note that $|F(z)|$ is $\SL_2(\mathbb{Z})$-invariant. We want to study the behaviour of the $L^p$-norm
\begin{equation*}
\Vert F\Vert_p:=\Big(\int_{\X}|F(z)|^p\dd \mu(z)\Big)^{\frac{1}{p}}
\end{equation*}
as $k$ tends to infinity.

Under the $L^2$-normalization $\Vert F\Vert_2^2=\vol(\X)=\pi/3$, a famous problem is to study the local $L^2$-norm of $F$ as $k$ becomes large, which is called the holomorphic quantum unique ergodicity (hQUE).
Specifically, for any fixed compact domain $\Omega$ of $\X$ whose boundary $\partial\Omega$ has hyperbolic measure zero, we have
\begin{equation}\label{eqn:hQUE}
  \frac{1}{\vol(\Omega)}\int_{\Omega}|F(z)|^2\dd\mu(z)\rightarrow 1, \quad \text{ as } k\to +\infty.
\end{equation}
This was first proposed by Rudnick and Sarnak \cite{RS94QUE}, and was settled by Holowinsky and Soundararajan \cite{HS10QUE}. 

When $p=4$, Blomer, Khan, and Young \cite{BKY13} gave a conjecture on the global $L^4$-norm:
\begin{equation}\label{eqn:L^4norm}
  \frac{1}{\vol(\X)}\int_{\X}|F(z)|^4\dd\mu(z)\rightarrow 2, \quad \text{ as } k\to +\infty.
\end{equation}
Equivalently, $\Vert F\Vert_4^4$ should be asymptotic to $2\pi/3$ as $k\to+\infty$. Unconditionally, they proved the upper bound $\Vert F\Vert_4^4\ll k^{1/3+\varepsilon}$ for any $\varepsilon>0$.
Under GRH, Zenz \cite{Zenz23} established the optimal bound $\Vert F\Vert_4^4\ll 1$ which is the best possible up to a constant factor toward \eqref{eqn:L^4norm}.  
In analogy with the $L^4$-norm problem, Huang \cite{Huang24} proved, under GRH and GRC, an asymptotic formula for a $(2, 2)$ joint distribution. 

For $p\geq 6$, Blomer, Khan, and Young also showed that 
\begin{equation}\label{eqn:L^plowerbound}
\Vert F\Vert_p\gg k^{\frac{1}{4}-\frac{3}{2p}-\varepsilon},
\end{equation}
which confirms that the $L^p$-norm blows up as $k\to +\infty$ when $p>6$. At the endpoint $p=\infty$, Xia \cite{Xia07} determined the sup-norm bound $\Vert F\Vert_\infty=k^{\frac{1}{4}+o(1)}.$

\medskip

One may ask how the global $L^6$-norm grows. 
By interpolation between the sharp $L^4$-norm bound and the sup-norm bound, we get $\Vert F\Vert_6\ll k^{1/12+\varepsilon}$.
Motivated by the $L^4$-norm conjecture, a natural question is whether the following asymptotic formula holds:
\begin{equation}\label{eqn:L^6normAsymp}
  \frac{1}{\vol(\X)}\int_{\X}|F(z)|^6\dd\mu(z)\rightarrow 6, \quad \text{ as } k\to +\infty.
\end{equation} 
Here the constant $6$ comes from the sixth moment of the complex Gaussian distribution $\mathcal{CN}(0,1),$ i.e. $\mathbb{E}[|Z|^6]=6$ for $Z\sim\mathcal{CN}(0,1).$
This choice is consistent with the fourth moment case in \eqref{eqn:L^4norm}.
Such a phenomenon is more naturally formulated in terms of local $L^p$-norms.
Specifically, for a fixed compact domain $\Omega$ as above in the hQUE setting, we expect the following asymptotic:
\begin{equation}\label{eqn:hRWC}
\frac{1}{\vol(\Omega)}\int_{\Omega}|F(z)|^{2m}\dd\mu(z)\rightarrow m!, \quad \text{ as } k\to +\infty,
\end{equation}
for any positive integer $m$ (see e.g. \cite{Huang24}). This is motivated by a holomorphic analog of the random wave model.

Although \eqref{eqn:L^6normAsymp} is suggested by the Gaussian moment heuristic, we demonstrate that it is false.
This note aims to disprove \eqref{eqn:L^6normAsymp} by establishing the following result.

\begin{theorem}\label{thm:L^6norm}
Let $k\geq 10$ be a large even integer. Then we have
\[
\max_{f\in H_k}\Vert F\Vert_6\gg (\log\log k)^{\frac{1}{2}}.
\]
\end{theorem}

Consider the spectral decomposition of $f^2$ in $H_{2k}$:
\begin{equation}\label{eqn:Hecke_expansion}
f^2(z)=\sum_{g\in H_{2k}}\frac{\langle F^2, G\rangle}{\vol(\X)} g(z).
\end{equation}
We have
\begin{equation}\label{eqn:spec_dec}
\Vert F\Vert_6^6=\sum_{g_1\in H_{2k}}\sum_{g_2\in H_{2k}}
\frac{\langle F^2, G_1\rangle\overline{\langle F^2, G_2\rangle}}{\vol(\X)^2}\int_{\X}G_1(z)\overline{G_2(z)}|F(z)|^2\dd \mu(z).
\end{equation}
Restricting to the diagonal terms ($g_1=g_2$) in \eqref{eqn:spec_dec}, then using Huang's asymptotic formula \cite[Theorem 1.5]{Huang24} and  Parseval's identity
$\Vert F\Vert_4^4=\frac{1}{\vol(\X)}\sum_{g\in H_{2k}}|\langle F^2, G\rangle|^2$, we have,  under GRH and GRC,
\[
\sum_{g\in H_{2k}}\frac{|\langle F^2, G\rangle|^2}{\vol(\X)^2}\int_{\X}|G(z)|^2|F(z)|^2\dd \mu(z)=\Vert F\Vert_4^4\Big(1+O((\log k)^{-1/4+\varepsilon})\Big),
\]
for any $\varepsilon>0.$
Moreover, Zenz's upper bound \cite[Theorem 1.1]{Zenz23} implies that the total diagonal contribution is bounded. 
Consequently, Theorem \ref{thm:L^6norm} shows that the contribution of the off-diagonal terms in \eqref{eqn:spec_dec} can dominate. 
This also indicates that the matrix coefficients in \eqref{eqn:spec_dec} do not exhibit strong cancellation for all $f\in H_k$.

The proof of Theorem \ref{thm:L^6norm} follows from a simple refinement of the argument proving \eqref{eqn:L^plowerbound} in \cite[\S 3]{BKY13}. Indeed, we can show that
\begin{equation*}
\Vert F\Vert_p\gg_{p} k^{\frac{1}{4}-\frac{3}{2p}}L(1,\sym^2f)^{-\frac{1}{2}}.
\end{equation*}
For $p=6$, we see $\Vert F\Vert_6\gg {L(1,\sym^2f)^{-\frac{1}{2}}}.$
The lower bound for the maximal $L^6$-norm is obtained from the extreme values of $L(1,\sym^2f)$ in the weight aspect. See Lau and Wu \cite{LW06}.

By the same method, we also establish a lower bound for the maximal joint mass.
Let $a,b\in \mathbb{Z}_{\geq 1}$ with $a+b=6$, and let $k,\ell\geq 2$ be even integers.
For $f\in H_{k}$ and $g\in H_{\ell}$, we define the $(a, b)$ joint mass by
\[
\Ma_{a, b}(f, g):=\int_{\X}|F(z)|^a|G(z)|^b\dd\mu(z),
\] 
where $F=y^{\frac{k}{2}}f$ and $G=y^{\frac{\ell}{2}}g.$
When the weights $k$ and $\ell$ are sufficiently close, we show that $\Ma_{a,b}(f,g)$ cannot remain uniformly bounded.

\begin{theorem}\label{thm:mixednorm}
Let $k$ and $\ell$ be two large even integers satisfying $|k-\ell|\ll \sqrt{k}$. Then we have
\[
\max_{f\in H_k\atop g\in H_\ell}\Ma_{a, b}(f, g)\gg (\log\log k)^{3},
\]
for any $a, b\in \mathbb{Z}_{\geq 1}$ with $a+b=6.$
\end{theorem}

When $(a, b)=(3,3)$, with $k$ and $\ell$ close, similar to the $L^6$-norm case,
$\Ma_{a, b}(f, g)$ does not obey the joint Gaussian moment asymptotic. 
This can be compared with the $(2, 2)$ joint distribution in \cite{Huang24} and the $L^4$-norm.

Among these joint masses, one of the most interesting cases is $(a,b)=(2,4).$
When $k$ is large and $\ell$ is fixed,  $\Ma_{2, 4}(f, g)$ behaves like a weighted $L^2$-norm for $f.$
And when $k$ is fixed and $\ell$ is large, $\Ma_{2, 4}(f, g)$ behaves like a weighted $L^4$-norm for $g.$
Both cases lie far outside the range of assumptions in Theorem \ref{thm:mixednorm}.
A natural heuristic suggests that, in these extreme cases, one should have
\[
\frac{\Ma_{2,4}(f,g)}{\vol(\X)}
\sim
\frac{\Vert F\Vert_2^2}{\vol(\X)}
\cdot
\frac{\Vert G\Vert_4^4}{\vol(\X)}
\sim 2.
\] 
In \S \ref{sec:M24}, we study the size of $\Ma_{2, 4}(f, g)$ in these extreme cases. 

Together with Theorem \ref{thm:mixednorm}, this shows that the $(2,4)$ joint mass suggests a phase transition phenomenon. 
While the extreme regimes suggest asymptotic independence, the near-diagonal regime $k\approx \ell$ reveals a correlation between the two Hecke eigenforms, leading to unexpectedly large values of $\Ma_{2, 4}(f,g).$
It is therefore interesting to understand when this change of behavior occurs.

\section{$L^6$-norm and joint mass of degree $6$}
\subsection{Proof of Theorem \ref{thm:L^6norm} and Theorem \ref{thm:mixednorm}}
We begin with the Fourier expansion of $f\in H_k$:
\[
f(z)=\rho_f(1)\sum_{n\geq 1}\lambda_f(n)(4\pi n)^{\frac{k-1}{2}}e(nz),
\]
where $\lambda_f(n)$ is the $n$-th Hecke eigenvalue of $f$, and under $L^2$-normalization
\[
|\rho_f(1)|^2=\frac{\pi}{3}\cdot\frac{2\pi^2}{L(1, \sym^2f)\Gamma(k)}.
\]
Similarly, we have the same formula for $g\in H_\ell.$
For $a, b\in \mathbb{Z}_{\geq 1}$ with $a+b=6,$
it follows that
\begin{equation*}
\begin{split}
 \Ma_{a, b}(f, g)
  &\geq \int_{1}^{\infty}\int_{0}^{1}|f(x+iy)y^{\frac{k}{2}}|^a|g(x+iy)y^{\frac{\ell}{2}}|^b\frac{\dd x\dd y}{y^2}\\
  &\geq \int_{1}^{\infty}\Big|\int_{0}^{1}f(x+iy)^ag(x+iy)^be(-(a+b)x)\dd x\Big|y^{\frac{ak+b\ell}{2}}\frac{\dd y}{y^2}\\
  &=\int_{1}^{\infty}\Big|\rho_f(1)(4\pi)^{\frac{k-1}{2}}e^{-2\pi y}y^{\frac{k}{2}}\Big|^a\Big|\rho_g(1)(4\pi)^{\frac{\ell-1}{2}}e^{-2\pi y}y^{\frac{\ell}{2}}\Big|^b\frac{\dd y}{y^2}\\
  &\gg \frac{1}{L(1, \sym^2f)^\frac{a}{2}L(1, \sym^2g)^\frac{b}{2}}
  \int_{1}^{\infty}\Big|\frac{e^{-2\pi y}(4\pi y)^{\frac{k}{2}}}{\sqrt{\Gamma(k)}}\Big|^a\Big|\frac{e^{-2\pi y}(4\pi y)^{\frac{\ell}{2}}}{\sqrt{\Gamma(\ell)}}\Big|^b\frac{\dd y}{y^2}\\
  &=\frac{(4\pi)^{\frac{ka+\ell b}{2}}}{L(1, \sym^2f)^\frac{a}{2}L(1, \sym^2g)^\frac{b}{2}\Gamma(k)^\frac{a}{2}\Gamma(\ell)^\frac{b}{2}}
  \int_{1}^{\infty}e^{-12\pi y}y^{\frac{ka+\ell b}{2}-2}\dd y.
\end{split}
\end{equation*}
The incomplete gamma integral satisfies
\[
\int_{1}^{\infty}e^{-12\pi y}y^{\frac{ka+\ell b}{2}-2}\dd y=
\int_{0}^{\infty}e^{-12\pi y}y^{\frac{ka+\ell b}{2}-2}\dd y-O(1)
\gg \frac{\Gamma(\frac{ka+\ell b}{2}-1)}{(12\pi)^{\frac{ka+\ell b}{2}-1}}.
\]
Thus we get
\[
 \Ma_{a, b}(f, g)\gg \frac{C_{a, b}(k, \ell)}{L(1, \sym^2f)^\frac{a}{2}L(1, \sym^2g)^\frac{b}{2}},
\]
where
\[
C_{a, b}(k, \ell)=\frac{\Gamma(\frac{ka+\ell b}{2}-1)}{3^{\frac{ka+\ell b}{2}}\Gamma(k)^\frac{a}{2}\Gamma(\ell)^\frac{b}{2}}.
\]
By Stirling's formula, for $k, \ell$ large with $|k-\ell|\ll \sqrt{k}$, 
\[
C_{a, b}(k, \ell)=\frac{1}{6\pi \sqrt{3}}\exp\Big(-\frac{ab}{24}\frac{(\ell-k)^2}{k}\Big)\Big(1+O(k^{-1/2})\Big)\gg 1.
\]
Hence we get
\begin{equation}\label{eqn:mixedlowerbound}
 \Ma_{a, b}(f, g)\gg \frac{1}{L(1, \sym^2f)^\frac{a}{2}L(1, \sym^2g)^\frac{b}{2}}.
\end{equation}

Now we use an important result of Lau and Wu \cite[Theorem 2]{LW06}, which asserts that there exists an $f\in H_k$ such that for $k\to \infty$,
\begin{equation}\label{eqn:L1largevalue}
L(1, \sym^2f)\leq (1+o(1))(e^\gamma\zeta(2)^{-2}\log\log k)^{-1},
\end{equation}
where $\gamma$ is the Euler constant. 
This implies that
\[
\max_{f\in H_k}\frac{1}{L(1, \sym^2f)}\gg \log\log k.
\]
Also we have the same bound for $L(1, \sym^2g)^{-1}.$
Combining these estimates with \eqref{eqn:mixedlowerbound} completes the proof of Theorem \ref{thm:mixednorm}.
By setting $k=\ell$ and $f=g$ in \eqref{eqn:mixedlowerbound} and then using \eqref{eqn:L1largevalue}, we get Theorem \ref{thm:L^6norm}.

\subsection{Asymptotic phenomenon for $\Ma_{2, 4}(f, g)$}\label{sec:M24}
Let $f\in H_k$ and $g\in H_\ell$. By Theorem \ref{thm:mixednorm}, when $k$ and $\ell$ are close, the $(2,4)$ joint mass fails to converge uniformly. In this section, we study the behaviour of $\Ma_{2,4}(f,g)$ in the complementary case where $k$ and $\ell$ are far apart. Indeed, we show that $\Ma_{2,4}(f,g)$ does not exhibit large values in the extreme range where the two weights are widely separated.

When $k$ is fixed and $\ell$ becomes large,  we get
\[
\Ma_{2, 4}(f, g)\sim \Vert G\Vert_4^4
\]
under the smooth $L^4$-norm conjecture \eqref{eqn:hRWC}.

When $k$ is large and $\ell$ is fixed (or is less than a small power of $\log k$), we essentially have 
\[
\Ma_{2, 4}(f, g)\sim \Vert G\Vert_4^4
\]
 by adapting the ideas used to prove hQUE. In fact, this asymptotic formula holds in a larger range.

\begin{proposition}\label{prop:kLarge}
Assume GLH. Then for $1\ll \ell \ll k^{1/3-\delta}$ with some fixed $\delta>0$, we have
\[
\Ma_{2, 4}(f, g)= \Vert G\Vert_4^4+O(\ell^{3/2}k^{-1/2+\varepsilon}),
\]
for any $\varepsilon>0.$
If we additionally assume the $L^4$-norm conjecture \eqref{eqn:L^4norm}, we have
\[
\Ma_{2, 4}(f, g)\rightarrow 2\vol(\X), \quad \text{ as } \ell \to \infty.
\]
\end{proposition}
\begin{proof}
Using spectral decomposition for $g^2$ in $H_{2\ell}$, we have
\begin{equation}\label{eqn:spec_dec1}
\Ma_{2, 4}(f, g)=\sum_{g_1\in H_{2\ell}}\sum_{g_2\in H_{2\ell}}
\frac{\langle G^2, G_1\rangle\overline{\langle G^2, G_2\rangle}}{\vol(\X)^2}\int_{\X}G_1(z)\overline{G_2(z)}|F(z)|^2\dd \mu(z)
\end{equation}
Then the spectral decomposition for $G_1\overline{G_2}$ in $L^2(\X)$ gives
\begin{equation}\label{eqn:g_1g_2decomp}
G_1(z)\overline{G_2(z)}=\delta_{g_1=g_2}+\sum_{j\geq 1}\langle G_1\overline{G_2}, \phi_j\rangle\phi_j(z)+\frac{1}{4\pi}\int_{\mathbb{R}}
\langle G_1\overline{G_2}, E_t\rangle  E_t(z)\dd t,
\end{equation}
where $\{\phi_j\}_{j\geq 1}$ is an orthonormal basis of Hecke--Maass forms on $\X$, and $E_t(z):=E(z, 1/2+it)$ is the Eisenstein series.
Converting the diagonal contribution ($g_1=g_2$) in \eqref{eqn:spec_dec1} into $L^4$-norm of $G$, we get
\begin{equation*}
\Ma_{2, 4}(f, g)=\sum_{g_1\in H_{2\ell}}\frac{|\langle G^2, G_1\rangle|^2}{\vol(\X)}\cdot \frac{\langle|G_1|^2, |F|^2\rangle}{\vol{(\X)}}+\mathcal{E},
\end{equation*}
where by \eqref{eqn:g_1g_2decomp},
\begin{multline*}
\mathcal{E}=\mathop{\sum\sum}_{g_1, g_2\in H_{2\ell}\atop g_1\neq g_2}
\frac{\langle G^2, G_1\rangle\overline{\langle G^2, G_2\rangle}}{\vol(\X)}\sum_{j\geq 1}\langle G_1\overline{G_2}, \phi_j\rangle\langle \phi_j, |F|^2\rangle\\
+\frac{1}{4\pi}\mathop{\sum\sum}_{g_1, g_2\in H_{2\ell}\atop g_1\neq g_2}
\langle G^2, G_1\rangle\overline{\langle G^2, G_2\rangle}\int_{\mathbb{R}}
\langle G_1\overline{G_2}, E_t\rangle \langle E_t, |F|^2\rangle\dd t.
\end{multline*}
By the computation in \cite[\S 3]{Huang24}, under GLH, for $\ell\ll k^{1-\varepsilon}$, we have
\[
\frac{\langle|G_1|^2, |F|^2\rangle}{\vol{(\X)}}=1+O(\ell^{1/2}k^{-1/2+\varepsilon})
\]
and also $\Vert G \Vert_4^4=\sum_{g_1\in H_{2\ell}}\frac{|\langle G^2, G_1\rangle|^2}{\vol(\X)} \ll \ell^\varepsilon.$
Hence the diagonal contribution is
\[
\Vert G \Vert_4^4+O(\ell^{1/2}k^{-1/2+\varepsilon}).
\]
Thus it suffices to show that $\mathcal{E}\ll \ell^{3/2}k^{-1/2+\varepsilon}$.
Note that Watson's formula and the Rankin--Selberg theory give (see \cite[\S 2.3]{BKY13}), for even $\phi_j$ and $f_1, f_2\in H_k, h\in H_{2k}$, 
\[
\begin{split}
|\langle F_1F_2, H\rangle |^2&=\frac{\Lambda(1/2, f_1\times f_2\times h)}{4\Lambda(1, \sym^2f_1)\Lambda(1, \sym^2f_2)\Lambda(1, \sym^2h)}\\
&=\frac{\pi^3}{4k-2}\cdot\frac{L(1/2, f_1\times f_2\times h)}{L(1, \sym^2f_1)L(1, \sym^2f_2)L(1, \sym^2h)},\\
|\langle F_1\overline{F_2}, \phi_j\rangle |^2&=\frac{\Lambda(1/2, f_1\times \overline{f_2}\times \phi_j)}{8\Lambda(1, \sym^2f_1)\Lambda(1, \sym^2f_2)\Lambda(1, \sym^2\phi_j)}\\
&=2\cdot\frac{L(1/2, f_1\times \overline{f_2}\times \phi_j)}{L(1, \sym^2f_1)L(1, \sym^2f_2)L(1, \sym^2\phi_j)}\mathcal{G}(k, t_j),\\
|\langle F_1\overline{F_2}, E_t\rangle |^2&
=4\cdot\frac{|L(1/2+it, f_1\times \overline{f_2})|^2}{L(1, \sym^2f_1)L(1, \sym^2f_2)|\zeta(1+2it)|^2}\mathcal{G}(k, t),\\
\mathcal{G}(k, t)&:=\frac{\pi^3|\Gamma(k-1/2+it)|^2}{4\Gamma(k)^2},
\end{split}
\]
where $F_j=y^{k/2}f_j$ ($j=1, 2$) and $H=y^kh$.
By Stirling's formula, $\mathcal{G}(k, t)\sim\frac{\pi^3}{4k}\exp(-t^2/k)$ for $|t|\leq k^{2/3}$ and is exponentially small for $|t|>k^{2/3}.$
Using GLH and the fact that the $L$-values at $1$ contribute an arbitrarily small power of spectral parameters, due to the weight function $\mathcal{G}(2\ell, t)$, we may truncate the sums and integrals to the ranges $t_j, |t|\leq \ell^{1/2+\varepsilon}$, and
\[
\mathcal{E}\ll \mathop{\sum\sum}_{g_1, g_2\in H_{2\ell}\atop g_1\neq g_2}\frac{1}{\ell^{1-\varepsilon}}
\left(\sum_{t_j\ll\ell^{1/2+\varepsilon}}\frac{1}{k^{1/2-\varepsilon}\ell^{1/2}}+\int_{|t|\leq \ell^{1/2+\varepsilon}}\frac{\dd t}{k^{1/2-\varepsilon}\ell^{1/2}}+O(k^{-10})
\right)
\ll \ell^{3/2}k^{-1/2+\varepsilon}.
\]
This completes the proof.
\end{proof}

\section*{Acknowledgements}
The authors would like to thank Professor Yuk-Kam Lau for his helpful discussions and comments.
They also want to thank Professor Bingrong Huang for his constant encouragements.
Li thanks Alfr\'ed R\'enyi Institute of Mathematics for providing an excellent academic environment. He also thanks the support from CSC program.


\end{document}